\documentclass{article}

\usepackage{amsmath,amsfonts,amsthm,amssymb}
\usepackage{graphicx}

\newtheorem{df}{Definition}
\newtheorem{lem}{Lemma}
\newtheorem{prop}{Proposition}
\newtheorem{st}{Theorem}
\newtheorem{conj}{Conjecture}
\newtheorem{cor}{Corollary}

\newcommand{\qbinom}[2]{\genfrac{[}{]}{0pt}{}{#1}{#2}}

\title{The Volume Conjecture for Augmented Knotted Trivalent Graphs}
\author{Roland van der Veen}

\begin{document}

\maketitle

\begin{abstract}
\noindent We propose to generalize the volume conjecture to
knotted trivalent graphs and we prove the conjecture for all
augmented knotted trivalent graphs. As a corollary we find that for any link $L$ there is a link containing $L$ for which the volume conjecture holds.
\end{abstract}

\section{Introduction}
\label{sec.Introduction}

The volume conjecture proposes a relation between the colored Jones invariants of a knot and the simplicial volume of its complement. In the formulation of \cite{MurakamiMurakami} the precise statement is as follows. Note that we use the variable $A$ from skein theory instead of the $q$ used in \cite{MurakamiMurakami} (the variables are related by $A^4 = q$).
\begin{conj} 
\label{conj.Volconj} \textbf{Volume conjecture} \cite{Kashaev}, \cite{MurakamiMurakami}.
For any knot $K$ we have:
$$\lim_{N\to \infty}\frac{2\pi}{N}\log|J_N(K)(e^{\frac{\pi i}{2N}})| = \mathrm{Vol}(\mathbb{S}^3 - K)$$
\noindent where $J_N$ denotes the $N$--colored Jones invariant of $K$ and $\mathrm{Vol}$ is the simplicial volume.
\end{conj}

\noindent To gain more insight into this conjecture and to find simple examples where it holds true we seek to generalize the conjecture to the class of knotted trivalent graphs (KTGs) as defined in \cite{DThurston}. Roughly speaking a KTG is a thickened embedded graph that is allowed to have multiple edges and also edges without vertices, so that KTGs generalize framed knots and links.  

Before turning to general KTGs we first discuss the generalization of the volume conjecture to links. For links the above version of the volume conjecture does not hold, because it fails for many split links \cite{MMOTY} and it also fails in a more serious way for the Whitehead chains defined in \cite{vanderVeen}. 

For a split link (a link some of whose components can be separated from each other by a sphere in the complement) the normalization of the colored Jones invariant has to be adjusted slightly. For knots the colored Jones invariant was normalized by dividing by the unnormalized invariant of the unknot. If we use this normalization for a split link then the colored Jones invariant vanishes at the root of unity as was noted in \cite{MMOTY}. To avoid this problem we propose the following normalization. For a split link with $s$ split components we normalize by dividing by the unnormalized invariant of the $s$--fold unlink. With this normalization the normalized colored Jones invariant becomes multiplicative under distant union, see section \ref{sec.ColoredJones}. Since the simplicial volume is additive with respect to distant union it follows that using this normalization the volume conjecture is true for a split link if it holds for all its split components.

The above conjecture fails in a more serious way in the case of the Whitehead chains. For these links it was shown \cite{vanderVeen} that $J_{N}(e^{\frac{\pi i}{2N}}) = 0$ for all even $N$ but that the limit proposed in the volume conjecture is still valid when one restricts to odd colors $N$. In section \ref{sec.ColoredJones} we will argue that the sequence of even colors is special and that the same failure is not as likely to occur in any other subsequence.

The above motivates the following modification of the volume conjecture that we propose to call the $\mathrm{so(3)}$ volume conjecture. To the best knowledge of the author it still stands a chance to hold for all knots and links. 

\begin{conj} $\mathbf{\mathrm{so(3)}}$ \textbf{volume conjecture}
\label{conj.So3volconj}
\newline The following form of the volume conjecture holds for all knots and links $L$:

$$\lim_{N\to \infty}\frac{2\pi}{N}\log|J_N(L)(e^{\frac{\pi i}{2N}})| = \mathrm{Vol}(L)$$

\noindent where $N$ runs over the odd numbers only and $J_N$ is normalized as described above. 
\end{conj}
 
\noindent The name $\mathrm{so(3)}$ is chosen because we restrict ourselves to odd colors $N$, i.e. representations of the Lie algebra $\mathrm{so(3)}$ instead of $\mathrm{sl}(2)$. The restriction to odd $N$ is natural because Kashaev's original invariant for triangulated links in $\mathbb{S}^3$ was also defined for odd $N$ only, see condition (3.12) in \cite{Kash1}. One might also argue more generally that the odd colors correspond to the spherical representations of $\mathrm{sl}(2)$.

Now we would like to generalize the volume conjecture even further to the class of knotted trivalent graphs (KTGs). A motivation for this generalization is that such graphs show up naturally in the computation of the colored Jones invariant when one applies fusion. Another motivation is that very simple graphs such as planar graphs will have relatively simple Jones invariants and a complement that is easy to triangulate. Considering graphs in their own right will furthermore clarify the role of six-j symbols, since they are the $\mathrm{sl}(2)$-invariants of the tetrahedral graph. In order to obtain a volume conjecture in the case of a KTG we need to define both the colored Jones invariant of a KTG and the volume of a KTG.

The generalization of the colored Jones invariant to KTGs is fairly straightforward and is based on the Kauffman bracket, see section \ref{sec.ColoredJones}. The idea is to connect the three incoming Jones--Wenzl idempotents in a trivalent vertex in the only possible planar way. Alternatively one can think of a trivalent vertex as a Clebsch--Gordan injector of the representation on the incoming strand into the tensor product of the representations of the two outgoing strands. We need a slight extension of the usual formalism to deal with half twisted edges such as a M\"{o}bius band. It is well known that this procedure yields a Laurent polynomial when or KTG is a knot or a link. For general KTG's this will not be the case and we obtain an invariant that is a quotient of Laurent polynomials.

The definition of the volume of a KTG is more complicated and will be treated in detail in section \ref{sec.GeometryComplement}. Here we give a brief overview of the ideas involved. The boundary of the exterior of a graph is a handlebody so if the exterior is to be hyperbolic then the boundary cannot be a cusp but we can require it to be a totally geodesic boundary as in \cite{Frigerio2}. However very different graphs can have homeomorphic exteriors because the structure of edges and vertices is lost. To fix this we exclude annuli and tori around the edges from the boundary so that they become cusps and the remaining punctured spheres become a geodesic boundary. This version of the exterior will be called the outside of the graph. It is shown in \cite{Frigerio} that rigidity still holds for such structures provided that we use a system of closed curves on the boundary to keep track of where the cusps should be. 

To deal with non-hyperbolic graphs we can no longer use the simplicial volume as was done for knots and links. This is because it was shown in \cite{Jungreis} that the simplicial volume of a hyperbolic manifold with geodesic boundary does not agree with its hyperbolic volume when the boundary is non-empty. To get around this we use the JSJ--decomposition and define the volume as the sum of the volumes of the hyperbolic pieces in the decomposition. For links this definition is known to agree with the simplicial volume.

Having defined the colored Jones invariant and the volume of a KTG, the above statement of the $\mathrm{so(3)}$ volume conjecture also makes sense for KTGs. Indeed, we propose that with this interpretation of volume and Jones invariant Conjecture \ref{conj.So3volconj} should be true for all KTGs. 

\begin{conj} The $\mathrm{so(3)}$ volume conjecture holds for all knotted trivalent graphs.
\label{conj.KTG}
\end{conj}

\noindent To provide some evidence for this claim we will prove the $\mathrm{so(3)}$ volume conjecture for the class of augmented KTGs defined below. This will be the main purpose of the paper. 

To describe the construction of augmented KTGs and to organize the calculations it is convenient to have a way to generate all KTGs by simple operations that we define now. For now let us think of a KTG as a thickened embedding of a graph whose edges are ribbons and whose vertices are disks. A more detailed treatment can be found in section \ref{sec.KTGs}. 

\begin{figure}[htp]
\begin{center}
\includegraphics[width = 12 cm]{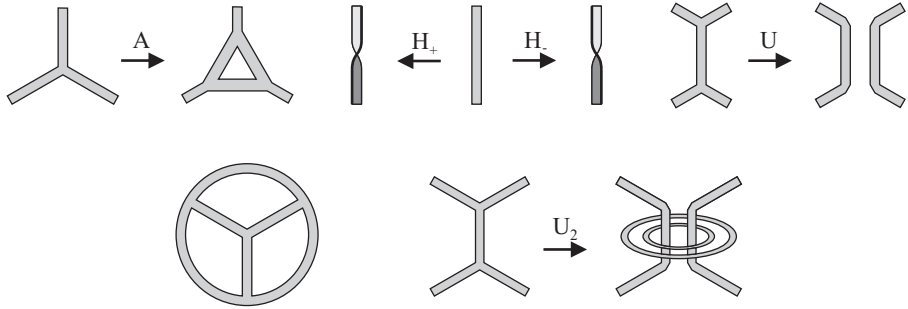}
\caption{First row: the four KTG moves triangle $A$, positive and negative half twists $H_\pm$ and Unzip $U$. Second row: the standard tetrahedron and the $n$--unzip $U_n$ (we have drawn the case $n = 2$).}
\label{fig.KTGmoves}
\end{center}
\end{figure}

\begin{df}
\label{df.KTGmoves}
The following four operations on KTGs will be called the KTG moves, see figure \ref{fig.KTGmoves}. The triangle move $A$ replaces a vertex by a triangle, the positive half twist move $H_+$ inserts a positive half twist into an edge, the negative half twist $H_-$ inserts a negative half twist and finally the unzip move $U$ takes an edge and splices it into two parallel edges.

We also define the following variations on the unzip move called the $n$--unzip $U_n$. This is the unzip together with the addition of $n$
parallel rings encircling the two unzipped strands.
\end{df}

\noindent The four KTG moves defined above are sufficient to generate all KTGs starting from the standard tetrahedron graph shown in figure \ref{fig.KTGmoves}.

\begin{st}
\label{st.KTGgen}
Any KTG can be obtained from the standard tetrahedron using the KTG moves only.
\end{st}

\noindent According to this theorem we can work with KTGs by studying sequences of KTG moves. Of course there are many inequivalent ways to produce the same KTG using the KTG moves, see section \ref{sec.KTGs}. 

Now we can define the notion of an augmented KTG.

\begin{df}
\label{df.Augmented}
Let $S$ be a sequence of KTG moves. Define the singly augmented KTG corresponding to $S$ to be the KTG obtained from the standard
tetrahedron by the moves of $S$ except that all unzip moves are to be
replaced by $1$--unzip moves. We will denote the singly augmented KTG corresponding to $S$ by $\Gamma'_S$.

Likewise the $n$-augmented KTGs corresponding to $S$ are defined
to be all the KTGs that can be produced from the standard
tetrahedron by the moves of $S$ except that every unzip
move is to be replaced by an $m$--unzip move, where $m \geq n$. Note that one may choose a different $m$ for all unzip moves in $S$.
\end{df}

\noindent Let $\Gamma_S$ be the KTG obtained from a sequence of KTG moves
$S$ and let $\Theta$ be an $n$--augmented KTG corresponding to $S$. Then $\Gamma_S$ is
contained in $\Theta$ and $\Theta - \Gamma_S$ is an $r$--fold unlink. Here is $r$
the number of rings that were added to $\Gamma_S$ to obtain the augmented KTG $\Theta$. The number
$r$ is called the number of augmentation rings of $\Theta$.

With all definitions in place we can now formulate the main theorem of this paper.
\newpage

\begin{st}
\label{st.Main} $\mathrm{\mathbf{(Main\ Theorem)}}$\newline
Let $S$ be a sequence of KTG moves. There exists an $n\in\mathbb{N}$ such that all
$n$-augmented KTGs $\Gamma$ corresponding to $S$ satisfy the following.
\begin{enumerate}

\item[1)]{Let $t$ be the number of
triangle moves in $S$ and let $r$ be the number of augmentation rings of $\Gamma$. Let $\theta$ be the number of half twists counted with sign and define the following numbers.
$$\phi_N = (-1)^{\frac{N-1}{2}}e^{\frac{N^2-1}{4N}\pi i} \quad \mathrm{and} \quad
\mathrm{sixj}_{N}=\sum_{k=0}^{\frac{N-1}{2}}\qbinom{\frac{N-1}{2}}{k}^4(e^{\frac{\pi i}{2N}})$$ The normalized $N$--colored Jones invariant of $\Gamma$ satisfies:
\[J_N(\Gamma)(e^{\frac{\pi i}{2N}}) = \left\{
    \begin{array}{ll}
        \phi_N^\theta N^{r}\mathrm{sixj}_N^{t+1} & \mbox{if $N$ is odd}\\
        0 & \mbox{if N is even}
    \end{array}
\right.\]}

\item[2)]{The JSJ--decomposition of the outside of $\Gamma$ consists of the outside of $\Gamma'_S$ and a Seifert
fibered piece for every $n$--unzip used in the construction of $\Gamma$ such that $n \geq 2$. It follows that $\mathrm{Vol}(\Gamma) = \mathrm{Vol}(\Gamma'_S)$

Moreover the outside of $\Gamma'_S$ is hyperbolic with geodesic boundary and can be
obtained explicitly by gluing $2t+2$ regular ideal octahedra.}

\item[3)]{$\Gamma$ satisfies the $\mathrm{so(3)}$ volume conjecture, but not the original volume conjecture.}
\end{enumerate}
\end{st}

\noindent The quantum binomial coefficients used in the above definition of $\mathrm{sixj}_N$ are defined in section \ref{sec.ColoredJones}. For a definition of the colored Jones invariant of a KTG, see section \ref{sec.ColoredJones}. The outside of a graph is defined in section \ref{sec.GeometryComplement}, it plays the role of the complement but it is a manifold with boundary pattern \cite{Matveev}. We will also define the volume of such manifolds. In section \ref{sub.GeometryAugmented} we will show how to obtain the explicit glueing of octahedra mentioned above.

The proof of parts 1) and 2) of the main theorem will be given in sections \ref{sec.ColoredJones} and \ref{sec.GeometryComplement}, but it is easy to see how part 3) follows from the first two parts. The key ingredient is the following observation about the numbers $\mathrm{sixj}_N$. It was shown in \cite{Costantino} that $\lim_{N\to \infty}\frac{2\pi}{N}\log|\mathrm{sixj}_N|=2\mathrm{Vol(Oct)}$, where $\mathrm{Vol(Oct)}$ means the hyperbolic volume of the regular ideal octahedron. Plugging in the formula for the colored Jones from part 1) gives: $$\lim_{N\to \infty}\frac{2\pi}{N}\log|J_N(\Gamma)(e^{\frac{\pi i}{2N}})| = 2(t+1)\mathrm{Vol(Oct)}$$ as a limit over all the odd numbers $N$. According to part 2) of the main theorem this is exactly the volume of $\Gamma$ since $\mathrm{Vol}(\Gamma) = \mathrm{Vol}(\Gamma'_S)$ and $\mathrm{Vol}(\Gamma'_S)$ equals $(2t+2)\mathrm{Vol(Oct)}$. The original volume conjecture does not hold because the the even values of $N$ give a colored Jones of $0$. This concludes the proof of part 3) assuming the first two parts of the main theorem.

Now let us note some immediate corollaries.

\begin{cor}
\label{cor.Sublink}
For every KTG $\Gamma$ there is a KTG $\Theta$ containing $\Gamma$ such that $\Theta-\Gamma$ is an
unlink and $\Theta$ satisfies the $\mathrm{so(3)}$ volume conjecture. If $\Gamma$ happens to be a link
then so is $\Theta$.
\end{cor}

\begin{cor}
\label{cor.Triangular}
The $\mathrm{so(3)}$ volume conjecture holds for all KTGs that can be constructed from the standard tetrahedron using the triangle move and the half twist move only. The original volume conjecture fails for such KTGs.
\end{cor}

\noindent The final corollary has nothing to do with the volume conjecture, but gives an alternative proof of a result by Baker \cite{Baker}.

\begin{cor}
\label{cor.Arithmetic}
Every link is a sublink of an arithmetic link.
\end{cor}
\begin{proof}
The singly augmented link corresponding to the given link is an arithmetic hyperbolic 3--manifold, since it is obtained from glueing regular ideal octahedra by symmetries of the tiling of hyperbolic space by regular ideal octahedra, see \cite{Thurston}.
\end{proof}

\noindent The organization of the paper is as follows. In section \ref{sec.KTGs} we discuss KTGs, KTG diagrams and KTG moves. The subject of section \ref{sec.ColoredJones} is skein theory. Here we define the colored Jones invariant of a KTG and show how it can be expressed in terms of six-j symbols. Specializing to the $N$--th root of unity and making use of the special properties of augmentation yields part 1) of the main theorem. In section \ref{sec.GeometryComplement} we give a definition of the volume of a 3--manifold with boundary and we study the geometry of the outside of an augmented KTG. Here we prove part 2) of the main theorem. Section \ref{sec.Conclusion} is a short summary and a conclusion. \newline

\noindent \textbf{Acknowledgement.} I would like to thank Dave Futer, Rinat Kashaev, Jessica
Purcell, Nicolai Reshetikhin and Dylan Thurston for enlightening
conversations and the organizers of the conferences, workshops and seminars in Hanoi, Strasbourg, Basel, Aarhus and Geneva for giving me the
opportunity to present parts of this work there.

\section{Knotted Trivalent Graphs}
\label{sec.KTGs}

In this section we state some general facts about knotted trivalent graphs (KTGs). We discuss the extra Reidemeister moves that are necessary to relate isotopic KTG diagrams and describe how every KTG can be generated from the standard tetrahedron by the KTG moves.

\begin{df}
A fat graph is a 1--dimensional simplicial complex together with an embedding into a surface as a spine.

A knotted trivalent graph (KTG) is a trivalent fat graph embedded
as a surface into $\mathbb{S}^3$ and considered up to isotopy.
\end{df}

\noindent By a diagram of a KTG we will mean a regular projection of its spine KTG onto the plane together with the usual crossing information and small diagonal lines indicating where an edge of a KTG makes a half twist. Except for the locations in the diagram where there is a half twist the surface of the KTG is assumed to be parallel to the projection plane as in the blackboard framing. The half twist pictures are necessary because a KTG such as the M\"{o}bius band cannot be given the blackboard framing. See Figure \ref{fig.ExampleKTG} for an example of a KTG together with its diagram.

\begin{figure}[here]
\begin{center}
\includegraphics[width = 12cm]{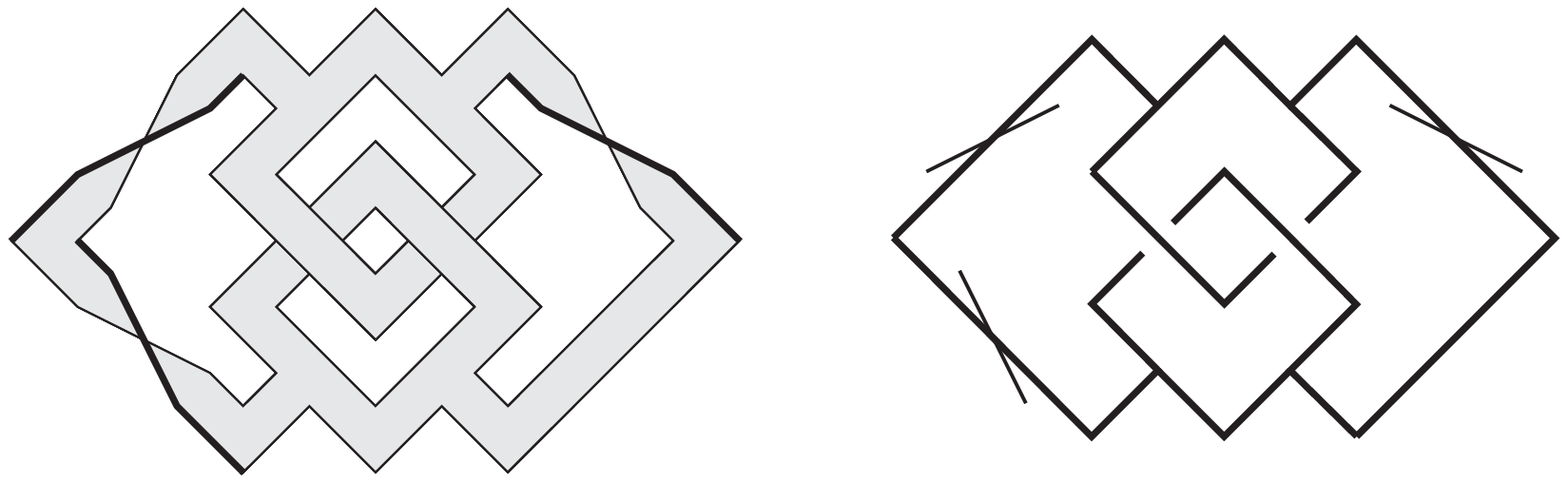}
\caption{A KTG and its diagram.}
\label{fig.ExampleKTG}
\end{center}
\end{figure}

Next we consider the moves that relate diagrams of isotopic KTGs. We will call these moves the trivalent isotopy moves. In addition to the usual Reidemeister moves for framed links we have moves related to the trivalent vertex and the
half-integral framing. These additional moves are called the fork slide, trivalent twist, twist slide and addition of twists, see figure \ref{fig.TrivalentIsotopy}.

\begin{figure}[here]
\begin{center}
\includegraphics[width = 12cm]{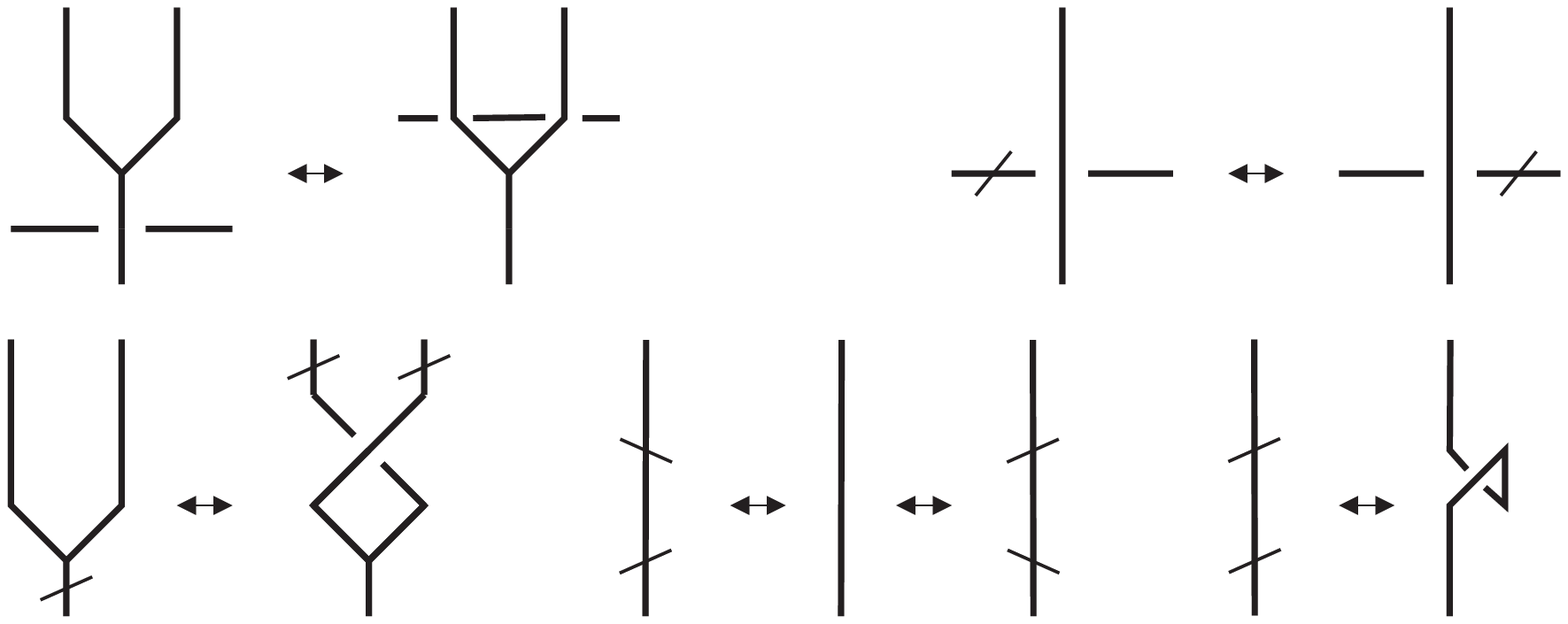}
\caption{The additional trivalent isotopy moves on a KTG diagram. First row: the fork slide and the twist slide. Second row: the trivalent twist and the addition of twists (multiple cases).}
\label{fig.TrivalentIsotopy}
\end{center}
\end{figure}

\begin{df}
\label{df.TrivalentIsotopy}
The trivalent isotopy moves are the Reidemeister moves for framed links and the following four moves on KTG diagrams: 
\begin{enumerate}
\item{Let the fork slide be the move where a strand is slid over or
under a trivalent vertex (first picture of figure \ref{fig.TrivalentIsotopy}).}
\item{One can slide a half twist past a crossing (second picture of figure \ref{fig.TrivalentIsotopy}). This is called the twist slide.} 
\item{The trivalent twist is the move where a single half twist is moved past a trivalent vertex. It starts on one edge, passes the vertex, creates a crossing and one half twist on the other two edges (third picture of figure \ref{fig.TrivalentIsotopy}). The sign of the initial half twist equals the sign of the crossing and the two ensuing half twists.}
\item{One may cancel or create two half twists of opposite sign on the same edge. Two half twists of equal sign on the same edge may be replaced by a curl of the same sign on that edge (last pictures of figure \ref{fig.TrivalentIsotopy}). This is called addition of twists.}
\end{enumerate}
\end{df}

\noindent The same arguments that are used in the proof of Reidemeister's theorem can be employed to prove the following theorem, see also \cite{Turaev}.

\begin{st}
\label{fig.TrivalentReidemeister}
Two KTG diagrams define isotopic KTGs if and only if the diagrams are related by trivalent isotopy moves.
\end{st}

\subsection{KTG moves}
\label{sub.KTGmoves}

We now take a closer look at the KTG moves defined in the introduction (Definition \ref{df.KTGmoves}). We will give a proof of Theorem \ref{st.KTGgen} that states that any KTG can be generated from the standard tetrahedron (see figure \ref{fig.KTGmoves}) using the KTG moves.

It is important to note that the result of an unzip move is determined by the number of half twists present on the edge. Technically such half twists have to be pushed off the edge before one can perform the unzip. In practice it is however much easier to remember that $n$ half twists on an edge give rise to $n$ crossings between the two parallel edges produced by the unzip. This follows from the trivalent isotopy moves defined above. Alternatively it can be checked physically by cutting a twisted band into two pieces along its core.

\begin{figure}[here]
\begin{center}
\includegraphics[width = 12 cm]{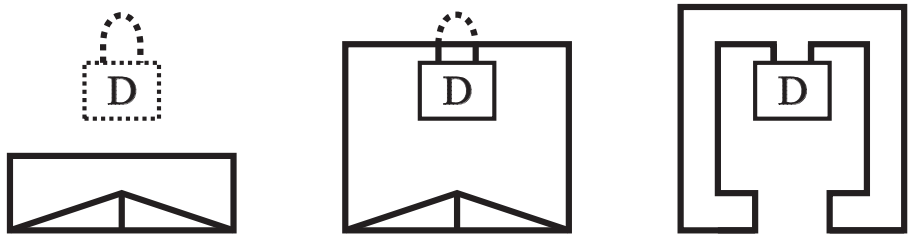}
\caption{Generating an arbitrary diagram $D$ from the tetrahedron by sweep-out.}
\label{fig.Sweepout}
\end{center}
\end{figure}

\begin{proof} (of Theorem \ref{st.KTGgen}).
We start with the diagram $D$ of the KTG that we want to generate drawn hatched in the first picture of figure \ref{fig.Sweepout}. Below it we draw a standard tetrahedron in black. The hatched part of the picture still needs to be generated and the black part is already done. 

We generate the diagram $D$ from the topmost edge of the tetrahedron step by step using the elementary steps depicted in figure \ref{fig.ElementarySteps}. The edge of the tetrahedron moves upwards over the hatched diagram $D$ and at every step we delete the hatched part of $D$ that is covered and regenerate it by one of the moves indicated in figure \ref{fig.ElementarySteps}. 

The elementary moves $A$, $H_{\pm}$ and $U$ in figure \ref{fig.ElementarySteps} are the KTG moves, and the moves $B$ and $C$ are a composition of KTG moves, see figure \ref{fig.Dewanagri} for a proof. The last step in the derivation of the move $C$ consists of unzipping the half twisted edge. To do this one can either cut the edge along its core or first isotope the half twist up to get a crossing.

\begin{figure}[here]
\begin{center}
\includegraphics[width = 12 cm]{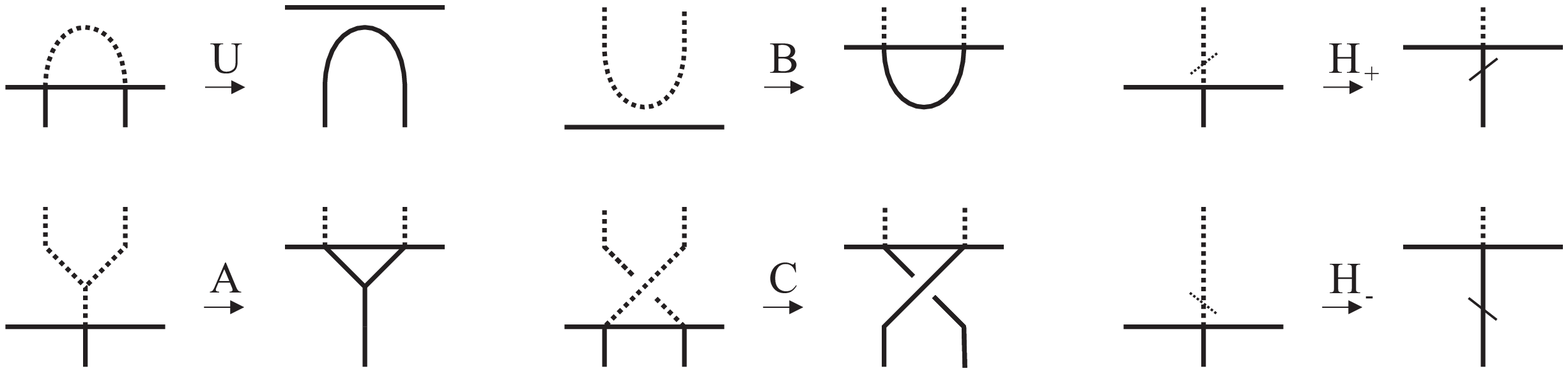}
\caption{The elementary steps encountered by the edge of the tetrahedron. The hatched parts are only meant to indicate the course of action, these parts are not actually there. With this in mind one recognizes the $U$ in the first picture as the unzip move.}
\label{fig.ElementarySteps}
\end{center}
\end{figure}

\noindent We stop the sweep-out process right before reaching the last hatched maximum of $D$, as indicated in the middle picture in figure \ref{fig.Sweepout}. To close the diagram we remove this maximum and unzip the three vertical edges of the tetrahedron to obtain the required diagram, see the last picture in figure \ref{fig.Sweepout}.

\begin{figure}[here]
\begin{center}
\includegraphics[width = 12 cm]{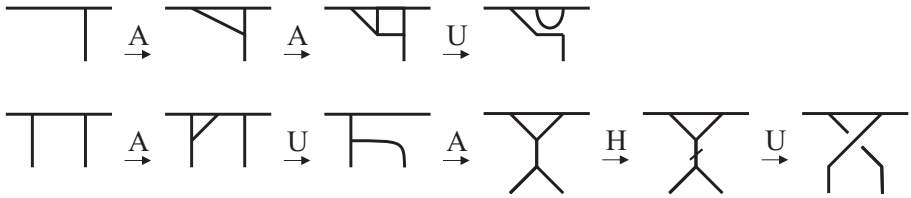}
\caption{A derivation of the move $B$ from the KTG moves (first row) and a derivation of $C$ from the KTG moves (second row).}
\label{fig.Dewanagri}
\end{center}
\end{figure}
\end{proof}

\noindent There are many ways to produce the same KTG using the KTG moves. For example if one starts with a single trivalent vertex and applies the triangle move then one can proceed in two ways to produce the same diagram. Either perform a single triangle move on the top vertex,
or do two triangle moves on the two lower vertices followed by an unzip on the middle edge at the bottom.

\section{The colored Jones invariant of a KTG}
\label{sec.ColoredJones}

Our definition and calculation of the colored Jones invariant will be based on the Kauffman bracket and its skein relation. We have chosen this language over the more general representation theoretic language because its formulas do not require a preferred direction in the projection plane. Throughout we will make use of the variable $A$ from skein theory. It is related to the $q$ from the introduction by $A^4 = q$.

\subsection{M\"{o}bius Skein Theory}
\label{sub.SkeinTheory}

To be able to include diagrams with half twisted edges we need to extend the usual skein theory a little. We propose to introduce the following extra relations called the half twist relations. A single edge with a positive half twist is equal to $(-A^3)^{1/2}$ times an untwisted edge. A single edge with a negative half twist is equal to $(-A^3)^{-1/2}$ times an untwisted edge. This definition is consistent with the value of the curl in ordinary skein theory and also with the trivalent isotopy move addition of twists from section \ref{sec.KTGs}.

\begin{figure}[here]
\begin{center}
\includegraphics[width = 12cm]{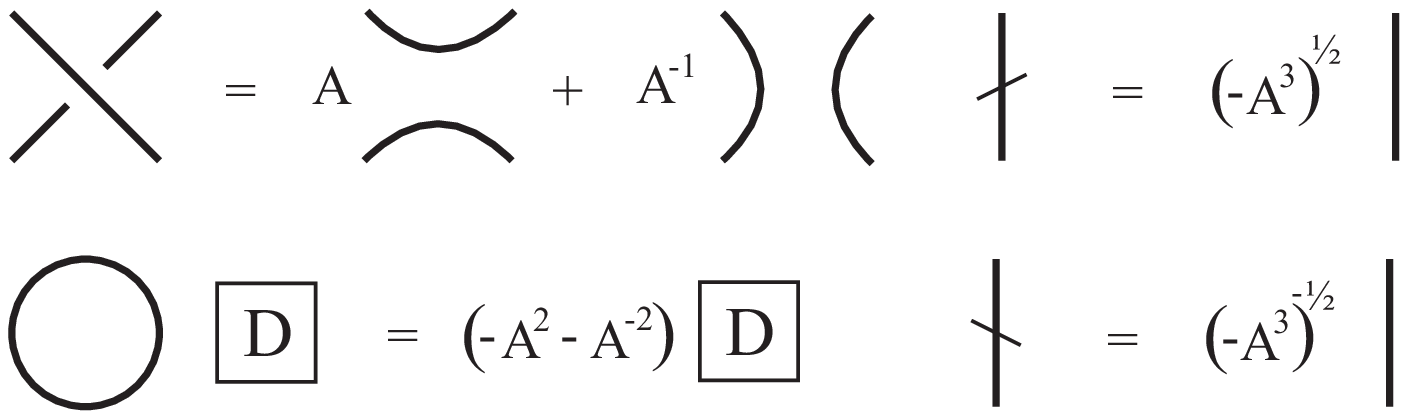}
\caption{The Kauffman relations and the additional twist relations together make up M\"{o}bius Skein Theory.}
\label{fig.Kauffman}
\end{center}
\end{figure}

\begin{df}
\label{df.Skein}
Let $\mathcal{R}$ be the quotient field of the ring of rational Laurent polynomials in $A^{1/2}$. Define the M\"{o}bius skein of a surface $\Sigma$ to be the $\mathcal{R}$-vector space of KTG diagrams without vertices in $\Sigma$ modulo the Kauffman bracket relations and the half twist relations shown in figure \ref{fig.Kauffman}.
\end{df}

\noindent The surface is allowed to have marked points on its boundary but in this case we only allow diagrams that have edges ending at all the boundary points.

Note that the above definition coincides with the usual definition of a skein space except for the half twist relations. A KTG diagram without vertices or half twists can be given the blackboard framing and its value in the M\"{o}bius skein will be exactly its value in the ordinary skein space.

We can now define the colored Jones invariant of a KTG using the notion of a Jones--Wenzl idempotent and a trivalent skein vertex, see \cite{MasbaumVogel}.

\begin{df} 
\label{df.ColoredJones}
Define the unnormalized $N$--colored Jones invariant $\langle \Gamma \rangle_N(A)$ of a KTG $\Gamma$ to be the Kauffman
bracket of the M\"{o}bius skein element obtained from a diagram of $\Gamma$ in the plane by
replacing every edge by $N-1$ parallel edges joined by a $N-1$--th Jones--Wenzl idempotent and every vertex by a trivalent skein vertex.

More generally we also define the bracket of a KTG diagram with integer labels on the edges to be the bracket of the skein element obtained by replacing an edge labeled $B$ by a $B-1$--th Jones--Wenzl idempotent and the vertices by the appropriate trivalent skein vertices. 
\end{df}

\noindent In this definition $\langle \Gamma \rangle_2$ coincides with the usual Kauffman bracket. As an example we note that if $M$ is the positive M\"{o}bius band then $\langle M \rangle_3 = -(A^8+A^4+1)$.

Note that replacing an $N$--colored edge with a half twist by parallel strands will cause the $N-1$ parallel edges to be intertwined and individually half twisted so that we get additional crossings and half twists. 

Since there is no planar way to connect an odd number of incoming edges, the trivalent vertex is defined to be zero when all edges have even colors. Therefore the colored Jones invariant of any KTG with at least one vertex is also zero for even $N$. In the next section we will see that at the $4N$--th root of unity this is the case for all augmented KTGs.

For the above definition to make sense we still have to prove that the value of $\langle \Gamma \rangle_N$ does not depend on the particular KTG diagram we choose for $\Gamma$. For this we first need a fairly standard lemma on the half twist, see also the last diagram in figure \ref{fig.Recoupling}.

\begin{lem}
\label{lem.HalfTwist}
A positive half twist on $n$ bands on top of an $n$--th Jones--Wenzl idempotent is equal to $(-1)^\frac{n}{2}A^{\frac{n(n+2)}{2}}$ times the untwisted bands with the same idempotent at the bottom. For the negative half twist we get $(-1)^{-n/2}A^{-n(n+2)/2}$ in the same way. 
\end{lem}
\begin{proof}
Because of the Jones--Wenzl idempotent there is only one way to resolve the crossings in the diagram that will give a non-zero contribution. A half twist on $n$ parallel bands produces $n(n-1)/2$ positive crossings yielding a contribution $A^{n(n-1)/2}$. Furthermore every strand contains a positive half twist, so the half twist relation gives another contribution of $(-1)^{n/2}A^{3n/2}$. Together this is exactly $(-1)^{n/2}A^{n(n+2)/2}$ as required. For the negative half twist the proof is the same.
\end{proof}

\begin{prop}
\label{prop.WellDefined}
The unnormalized $N$--colored Jones invariant $\langle \Gamma \rangle_N(A)$ of a KTG $\Gamma$ is a well defined invariant of KTGs.
\end{prop}

\begin{proof}
We need to check that the value of the unnormalized colored Jones invariant is unchanged under the trivalent isotopy moves of KTG diagrams defined in definition \ref{df.TrivalentIsotopy}. For the Reidemeister moves this is clear. Because a trivalent vertex is turned into a skein element without trivalent vertices or half twists, invariance under the fork slide move follows from invariance under Reidemeister II and III. 

Lemma \ref{lem.HalfTwist} proves the invariance of the Jones under the twist slide move and the addition of half twists. Invariance under the trivalent twist move now follows from this lemma in combination with Theorem 3 of \cite{MasbaumVogel}. 
\end{proof}

\noindent Note that the above proof also shows that the bracket of a KTG whose edges are colored by any integers is an invariant. This invariant is multiplicative under distant union.

To relate our definition of the unnnormalized colored Jones invariant to the ones that can be found in the literature we note that when $\Gamma$ is a link it coincides with $(-1)^{N-1}$ times the value of the unnormalized Jones invariant defined in \cite{Masbaum}. This follows from the remark that the bracket of a KTG diagram without half twists or vertices equals the bracket of the framed link in the usual skein theory.

The normalization of the Jones invariant that is used in the $\mathrm{so}(3)$ volume conjecture (Conjecture \ref{conj.So3volconj}) is defined as follows.

\begin{df}
\label{df.ColoredJonesNormalized}
Define the normalized colored Jones invariant of a KTG $\Gamma$ with $s$ split components to be $J_N(\Gamma) = \langle \Gamma \rangle_N/\langle U^s \rangle_N$, where $U^s$ is the $s$-component unlink.
\end{df}

\noindent For the volume conjecture we need to specialize to $A = \exp(\pi i/2N)$ but $\langle U^s \rangle_N = (-1)^{s(N-1)}[N]^s$, where $[N] = (A^{2N}-A^{-2N})/(A^{2}-A^{-2})$. At this value of $A$ we have $[N] = 0$ so we have check that we can divide out this pole and still get a well defined answer. 

\begin{prop}
\label{prop.WellDefinedNormalized}
The normalized $N$--colored Jones invariant has a well defined value at $A = \exp(\pi i/2N)$.
\end{prop}

\begin{proof}
Since the unnormalized colored Jones invariant is multiplicative under distant union, the normalized colored Jones invariant also has this property. Therefore we can assume that the number of split components of our KTG $\Gamma$ is $1$. Let $\Gamma$ be the closure of a 1--1 tangle $\Theta$. We label the edges of $\Theta$ with $N$ and interpret $\Theta$ as an element of the M\"{o}bius skein of a square with $2N-2$ marked boundary points. As in the Temperley--Lieb algebra we can now write $\Theta$ as a scalar $f_N(A)$ times the $N-1$th Jones--Wenzl idempotent. Closing the tangle $\Theta$ we find that $\langle \Gamma \rangle_N = \langle U \rangle_N f_N(A)$. 

It now remains to show that $f_N(A)$ is a quotient of Laurent polynomials in $A^{1/2}$ whose denominator is not zero at $A = \exp(\pi i/2N)$. To calculate $f_N(A)$ we expand all crossings and half twists in $\Theta$ so as to obtain an element of the Temperley--Lieb algebra and the component of the identity in this expression is $f_N(A)$. The calculation of $f_N(A)$ will involve the Jones--Wenzl idempotents, the skein relation and the half twist relations. From the recursive definition of the Jones--Wenzl idempotent it is clear that $f_N(A)$ is a quotient of Laurent polynomials in $A^{1/2}$ whose denominator does not have poles at $A = \exp(\pi i/2N)$. 
\end{proof}

\noindent It follows from our discussion that the normalized colored Jones invariant is multiplicative under both connected sum and distant union of KTG diagrams. To see the multiplicativity with respect to connected sum we observe that it corresponds to concatenation of 1--1 tangles and hence to multiplication of scalars.

We now move on to the problem of calculating the unnormalized colored Jones invariant of a general KTG. Theorem \ref{st.KTGgen} tells us that all KTGs can be constructed from the standard tetrahedron by applying the KTG moves. It turns out that in skein theory the KTG moves correspond to the well known formulas shown in figure \ref{fig.Recoupling}, see also \cite{MasbaumVogel}. We will show below that these formulas can be used to calculate the colored Jones polynomial of any KTG from a sequence of KTG moves generating it.

\begin{figure}[h]
\begin{center}
\includegraphics[width = 12 cm]{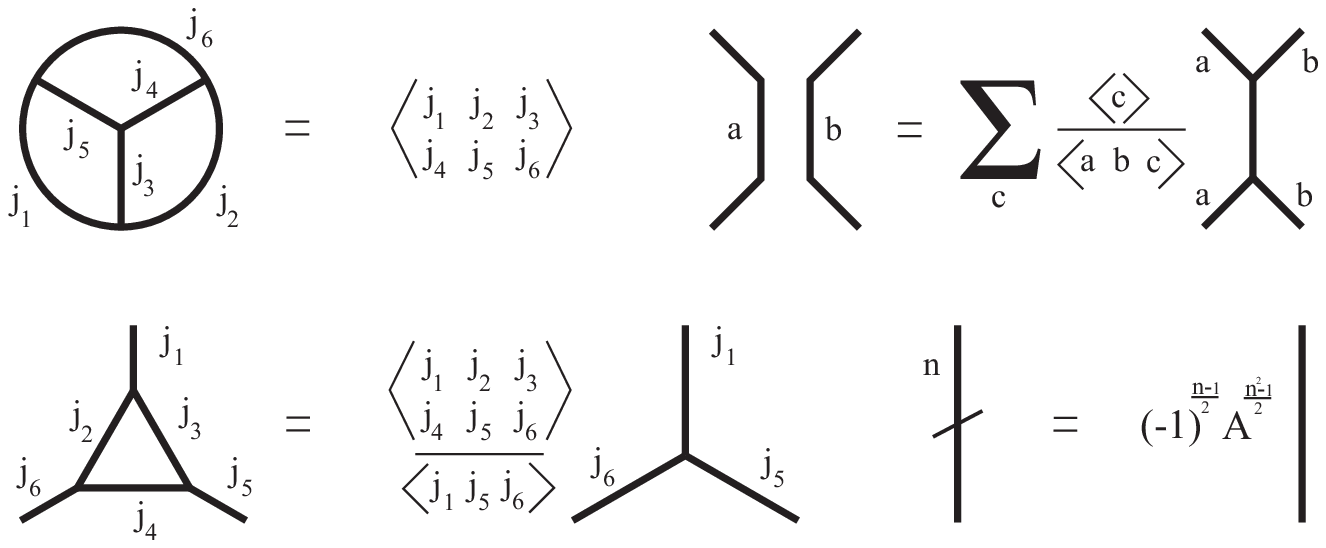}
\caption{The value of the skein of the labeled standard tetrahedron is the six-j symbol defined below. The fusion formula reverses the unzip move, the half twist formula undoes the half twist move and the triangle formula undoes the triangle move.}
\label{fig.Recoupling}
\end{center}
\end{figure}

To be able to write down the formulas for the six-j symbols shown in figure \ref{fig.Recoupling} we first recall the definition of a quantum integer $[n]= \frac{A^{2n}-A^{-2n}}{A^2-A^{-2}}$. The value of the unknot is $\langle U \rangle_N = \langle N \rangle = (-1)^{N-1}[N]$. Quantum factorials and binomial coefficients are defined in the usual way in terms of the quantum integers. 

Given six integer labels $j_1,...,j_6$ on the edges of a tetrahedron as in figure \ref{fig.Recoupling} such that all trivalent vertices are non-zero, define $V_1,V_2,V_3,V_4$ to be a half times the sums of the three labels around each of the four vertices. For example $V_1  = (j_1+j_2+j_3)/2$. Also define $\Box_1,\Box_2,\Box_3$ to be a half of the sums of the labels in the three squares (pairs of opposite edges). According to \cite{MasbaumVogel} the value of the tetrahedron is: $$\left\langle\begin{array}{ccc} j_1 & j_2 & j_3\\ j_4 & j_5 & j_6\end{array}\right\rangle \quad \mathrm{where} \quad \left\langle\begin{array}{ccc} j_1+1 & j_2+1 & j_3+1\\ j_4+1 & j_5+1 & j_6+1\end{array}\right\rangle =$$ $$ \frac{\prod_{m,n}(\Box_m-V_n)}{\prod_{k=1}^6 [j_k]!}\sum_{z=\max V_i}^{\min\Box_j}\frac{(-1)^z[z+1]!}{\prod_{r}(\Box_r-z)\prod_{s}(z-V_s)}$$

\noindent The value of the labeled theta graph is given by $\langle a\ b\ c \rangle$, where $$\langle a+1\ b+1\ c+1 \rangle = (-1)^{s}\frac{[s+1]![s-a]![s-b]![s-c]!}{[a]![b]![c]!} \quad \mathrm{where} \ s = \frac{a+b+c}{2}$$

\noindent The sum in the upper right equation in figure 8 ranges over all possible triples for which the trivalent vertex is nonzero, that is all $c$ such that $|a-b|\leq c \leq a+b$ and $a+b+c$ is odd. It should be remarked that since we replace an edge labeled $b$ by a $b-1$--th Jones--Wenzl idempotent while they are replaced by $b$--th Jones--Wenzl idempotents in \cite{MasbaumVogel} there is a slight shift of indices.  

The formulas in figure \ref{fig.Recoupling} suffice to give a formula for the colored Jones invariant of any KTG in terms of the six-j symbols. By theorem \ref{st.KTGgen} we know that any KTG $\Gamma$ can be constructed from the tetrahedron by a sequence $S$ of KTG moves. To calculate $\langle \Gamma \rangle_N$ we start with the diagram corresponding to $S$ and label all edges by $N$. Now we reverse the KTG moves in $S$ move by move. At every step we keep track of the newly produced edge labels in the diagrams that we get. The formulas in figure \ref{fig.Recoupling} tell us that we get a six-j symbol when we reverse the triangle move $A$, a summation with so called fusion coefficients from the unzip move $U$ and a factor from the half twist moves $H_{\pm}$. In the next subsection we will use this knowledge to calculate the colored Jones of an augmented KTG at the relevant root of unity.

Finally note that it is well known that the colored Jones invariant of knots and links is a Laurent polynomial. For KTGs this is generally not the case. The colored Jones invariant (normalized or not) of a $KTG$ is merely a quotient of Laurent polynomials in $A^{1/2}$. As an example let us calculate the normalized colored Jones invariant of the theta graph $\theta$. From the formula in figure \ref{fig.Recoupling} we get: $$J_N(\theta) = (-1)^{3k} \frac{[3k+1]![k]!^3}{[2k]!^3[2k+1]} \quad \quad N = 2k+1$$ 

\noindent By considering the zeros of the numerator and the denominator it is clear that this is not a Laurent polynomial for odd $N$ greater than $3$. For example we can look at the number of zeros at $A = \exp(i\pi/4k)$.

\subsection{Asymptotics and augmentation}
\label{sub.AsymptoticsAugmentation}

We have seen that the unnormalized $N$--colored Jones invariant takes the form of a multi-sum of products and quotients of
quantum integers. Every unzip contributes a summation with fusion
coefficients, every triangle move produces a six-j symbol and every
half twist move contributes a power of $A$. 

It is not trivial to determine the asymptotics of such a multisum formula. To circumvent this difficulty we augment the KTG. Adding extra unknotted ring-like components actually simplifies the sum at the relevant root of unity because of the following formula from skein theory \cite{Lickorish}, see figure \ref{fig.Ring}. The value of a $k-1$--th Jones--Wenzl idempotent encircled by a closed $N-1$--th idempotent is $(-1)^{N-1}[k N]/[k]$ times the idempotent. 

\begin{figure}[here]
\begin{center}
\includegraphics[width = 12cm]{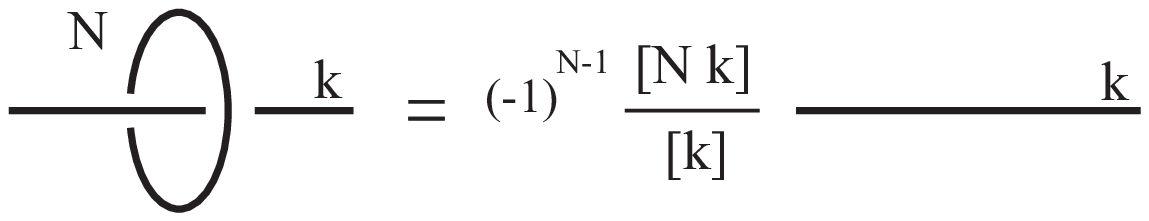}
\caption{The effect of adding a ring to a labeled edge. Note that every edge is replaced by a Jones--Wenzl idempotent.}
\label{fig.Ring}
\end{center}
\end{figure}

\noindent The following lemma gives a calculation of the above value at our root of unity. 

\begin{lem}
\label{lem.Ring}
In skein theory adding a ring labeled $N$ encircling an edge labeled $k$ is the same as multiplying the edge by $(-1)^{N-1}[k N]/[k]$. The value of this constant is
\[\lim_{A \to e^{\pi i/2N}}(-1)^{N-1}\frac{[k N]}{[k]} = \left\{
    \begin{array}{ll}
        (-1)^{N-1+k-k/N}N & \mbox{$\mathrm{if}$ $N \mid k$}\\
        0 & \mbox{$\mathrm{if}$ $N \nmid k$}
    \end{array}\right. \]

\end{lem}
\begin{proof}
The value of $(-1)^{N-1}[k N]/[k] = (-1)^{N-1} \frac{A^{2kN}-A^{-2kN}}{A^{2k}-A^{-2k}}$ at $A=e^{\pi i/2N}$ depends on whether or not the denominator vanishes. The numerator is always zero but the denominator is zero if and only if $N \mid k$, therefore the value is $0$ if $N$ does not divide $k$. Using 'l Hospital's rule we calculate the value in case $N \mid k$.
$$\lim_{A \to e^{\pi i/2N}}(-1)^{N-1}\frac{[k N]}{[k]} = \lim_{A \to e^{\pi i/2N}}(-1)^{N-1}\frac{2kNA^{-1}}{2kA^{-1}}\frac{A^{2kN}+A^{-2kN}}{A^{2k}+A^{-2k}} =$$ $$(-1)^{N-1}\frac{2(-1)^{k}}{2(-1)^{k/N}} N = (-1)^{N-1+k-k/N}N$$
\end{proof}

\noindent The above lemma suggests that we can use an edge with a ring as a kind of delta function. In other words we can try to pick only the term $k = N$ from a sum over edges labeled $k$ by adding a ring to the edge. This will turn the expression of the colored Jones invariant into a single term thus making an asymptotic analysis possible. To make this idea precise we need to be careful because of poles in the six-j symbols and the possibility of several multiples of $N$ dividing $k$. This is done in the proof of part 1) of the main theorem that we will now present.

\begin{proof} (of part 1) of the main theorem (theorem \ref{st.Main}))
Let us fix a sequence $S$ of KTG moves and let $\Theta$ be the KTG generated by $S$ starting from the tetrahedron. In the previous subsection we have seen that it is possible to express the colored Jones invariant of $\Theta$ in terms of the sequence $S$ and the formulas from figure \ref{fig.Recoupling} by reversing the KTG moves one by one until one reaches the tetrahedron. From the formulas in figure \ref{fig.Recoupling} one sees that the unnormalized colored Jones invariant can be written as a multisum of products and quotients of quantum integers. 

Let $n$ be a fixed integer that is at least one more than the maximum number of poles at $A=e^{\pi i/2N}$ in the summands of the expression of the unnormalized $N$ colored Jones of our KTG $\Theta$. It is very important to note that one can choose such a $n$ to be independent of $N$. To see this we write out all the six-j symbols in the expression for the colored Jones invariant to see that it is a multisum of quotients of quantum factorials. Moreover there is a number $a$ depending only on $S$ such that if $[r]$ occurs in a summand of the expression for the colored Jones then $r\leq aN$. Since the number of zeros of $[r]!$ at $A = \exp(2\pi i /2N)$ is $\left\lfloor r/N \right\rfloor$ we know that all terms $[r]!$ that occur have less than $a$ zeros. It follows that the number of poles in a summand of the multi-sum is less than $a$ times the number of quantum factorials present in the denominator. Suppose that the number of quantum factorials is at most $f$ then we can set $n = af+1$. Note that $f$ is independent of $N$ as well.

Now let $\Gamma$ be an $n$--augmented KTG. If we calculate the unnormalized colored Jones invariant then we get the same multisum as we did for $\Theta$ except that according to lemma \ref{lem.Ring} we have at least $n$ factors $(-1)^{N-1}(\frac{[k N]}{[k]})$ for every unzip move, where $k$ is the summation variable created by the formula for reversing the unzip in skein theory, see figure \ref{fig.Recoupling}. By lemma \ref{lem.Ring} and the construction of $n$ only those summands of the multisum for $\Gamma$ for which the summation variables are multiples of $N$ are non-zero at $A = \exp(2\pi i /2N)$.

Actually only the term where all summation variables are equal to $N$ is non-zero at the root of unity. To see this suppose that we have a term where one summation index equals $uN$ for some integer $u>1$. We may assume that the index whose value is $uN$ is the first in the order of appearance of the summations in the calculation. This means that the index what created at a stage of the calculation when multiples of $N$ other than $N$ itself did not occur. Since labels that are not multiples of $N$ will not contribute the only possibility is that the label came from fusing two edges labeled $N$. But this implies that the new summation ranges over the odd integers between $0$ and $2N$. Therefore only the summand where all labels are $N$ contributes.

Now that we know that in the multi-sum expression for the unnormalized colored Jones invariant of $\Gamma$ only the term where all indices are $N$ contributes at this root of unity, we can easily write down a closed form expression for its value. Reversing the KTG moves in $S$ now becomes a matter of multiplying by a particular factor. For the triangle move this factor is $\frac{\left\langle\begin{array}{ccc} N & N & N\\ N & N & N\end{array}\right\rangle}{\langle N N N \rangle}$, for the unzip it is $\frac{\langle N \rangle}{\langle N N N \rangle}$, for the half twist $H_\pm$ it is $(-1)^{\pm(N-1)/2}A^{\pm(N^2-1)/2}$ and finally one factor $\left\langle\begin{array}{ccc} N & N & N\\ N & N & N\end{array}\right\rangle$ for the tetrahedron.

Taking into account the normalization and the powers of $N$ from the augmentation we get the following formula for the normalized $N$--colored Jones invariant at $A = \exp(\pi i/2N)$. Note that $\Gamma$ has only one split component so that we divide by $\langle U \rangle_N$ only once. $$J_N(\Gamma)(e^{\frac{\pi i}{2N}}) = \left((-1)^{(N-1)/2}A^{\frac{N^2-1}{2}}\right)^\theta N^r\left(\frac{\langle N N N \rangle}{\langle N \rangle}\right)^u \times $$ $$ \frac{\left\langle\begin{array}{ccc} N & N & N\\ N & N & N\end{array}\right\rangle ^t}{\langle N N N \rangle^t}\frac{\left\langle\begin{array}{ccc} N & N & N\\ N & N & N\end{array}\right\rangle}{\langle N\rangle}$$
where $\theta$ is the number of half twists counted with sign, $u$ is the number of unzips in the sequence, $t$ the number of triangle moves and $r$ the number of augmentation rings.

Note that this formula is zero when $N$ is even, because then $$\left\langle\begin{array}{ccc} N & N & N\\ N & N & N\end{array}\right\rangle$$ is zero for generic $A$ because the trivalent vertices do not exist.

For odd $N = 2k+1$ we actually have $\frac{\langle N N N \rangle}{\langle N \rangle} = 1$ at $A = \exp(\pi i/2N)$. To see this, first observe that at this value of $A$ we have $[N+j] = -[j] = -[N-j]$. For generic values of $A$ we write: $$\frac{\langle N N N \rangle}{\langle N \rangle} = (-1)^{3k}\frac{[3k+1]![k]!^3}{[2k+1][2k]!^3} = $$
$$(-1)^k \frac{[1]\cdots[k][k+1]\cdots[2k][2k+1][2k+2]\cdots[3k+1]}{([k+1]\cdots[2k])^3[2k+1]}$$ $$ = (-1)^k \frac{[1]\cdots[k][2k+2]\cdots[3k+1]}{([k+1]\cdots[2k])^2}$$
At $A = \exp(\pi i/2(2k+1))$ this becomes equal to $1$ since $[1]\cdots[k] = [2k][2k-1]\cdots[k+1]$ and $[2k+2]\cdots[3k+1] = (-1)^k[1][2]\cdots[k]$.

The same type of calculation shows that: $$\frac{\left\langle\begin{array}{ccc} N & N & N\\ N & N & N\end{array}\right\rangle}{\langle N \rangle}(e^{\pi i/2N}) = \mathrm{sixj}_{N}$$
where $$\mathrm{sixj}_{N}=\sum_{k=0}^{(N-1)/2}\qbinom{(N-1)/2}{k}^4(e^{\pi i/2N})$$ 

\noindent Therefore the formula for the colored Jones of the KTG $\Theta$ at the root of unity simplifies to:

$$J_N(\Theta)(e^{\frac{\pi i}{2N}})\left((-1)^{(N-1)/2}A^{\frac{N^2-1}{2}}\right)^\theta N^r \mathrm{sixj}_N^{t+1}$$

\noindent as claimed in part 1) of the main theorem.
\end{proof}

\section{The geometry of the complement of an augmented KTG}
\label{sec.GeometryComplement}

In this section we are concerned with the definition and the calculation of the volume of a KTG. The generalization to graphs is not straight forward because of the following problem. A knot is determined by its complement but a graph is not. The homeomorphism type of the complement of a graph does not say much about the graph itself. For example the standard tetrahedron and the connected sum of two theta graphs, shown in figure 10 have homeomorphic complements. From the point of view of the volume conjecture this is very inconvenient because the colored Jones invariant at the root of unity does distinguish these graphs. If the volume conjecture is to hold for KTGs then we need to add a little structure to the complement so that we can recover the adjacency matrix of the graph from its complement. 

In the first subsection we will show how to assign a 3--manifold with boundary to any embedded graph such that the graph can be recovered from the 3--manifold and we still have the possibility of rigid hyperbolic structures. In the second subsection we apply these ideas to augmented KTGs very explicitly and we give a proof of the second part of the main theorem (theorem \ref{st.Main}).

\begin{figure}[here]
\begin{center}
\includegraphics[width = 12cm]{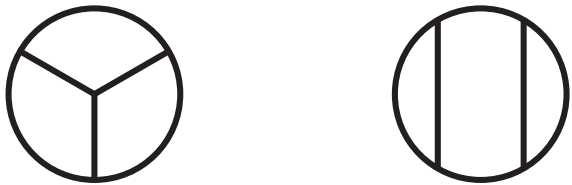}
\caption{Two KTGs with homeomorphic complements.}
\label{fig.HomeoKTGs}
\end{center}
\end{figure}

\subsection{The volume of a 3--manifold with boundary}

In this section we lay down the necessary foundations that allow us to define the hyperbolic volume of a graph in $\mathbb{S}^3$. We start with some general notions about hyperbolic structures on 3--manifolds with boundary following \cite{Frigerio}.

\begin{df}
\label{df.Hyp}
A 3--manifold $M$ is called a hyperbolic manifold with geodesic boundary if it is locally modeled on the right upper half space \newline $\{(x,y,z)\in \mathbb{H}^3 | x \geq 0 \}$.
\end{df}

\noindent In the next subsection we will construct many hyperbolic manifolds with geodesic boundary by glueing ideal polyhedra along some of their faces. The remaining faces will make up the boundary.

Mostow rigidity holds for finite volume hyperbolic 3--manifolds with geodesic boundary provided that the boundary is compact \cite{Frigerio} but when the boundary is non-compact then it may fail. However even in the case of non-compact boundary one can save the rigidity result by considering annular cusp loops. In order to define this notion we first sketch the construction of the natural compactification of a hyperbolic 3--manifold with geodesic boundary.

Let $M$ be an orientable, finite volume, hyperbolic 3--manifold with geodesic boundary. The double $D(M)$ of $M$ is hyperbolic without boundary. Therefore it consists of a compact portion together with some cusps based on Euclidean surfaces. It follows that $M$ also consists of a compact portion together with some cusps of the form $T\times[0,\infty)$, where $T$ is a Euclidean surface with geodesic boundary such that $(T\times[0,\infty))\cap \partial T = \partial T \times [0,\infty)$. $M$ now admits a natural compactification $\bar{M}$ by adding such a surface $T$ for each cusp. Note that the compactification $\bar{M}$ of $M$ is obtained by adding tori and closed annuli. The set of these annuli will be called $\mathcal{A}_M$.

\begin{df}
\label{df.CuspLoop}
A loop $\gamma$ in a hyperbolic 3--manifold with geodesic boundary $M$ is called an annular cusp loop if in $\bar{M}$ it is freely homotopic to the core of an annulus of $\mathcal{A}_M$.
\end{df}

\noindent With this notion in place we can state the rigidity theorem for hyperbolic 3--manifolds with boundary proven in \cite{Frigerio}.

\begin{st}
\label{st.Rigidity}
Let $M$ and $M'$ be two orientable finite volume hyperbolic 3--manifolds with geodesic boundary and let $\phi:\pi_1(M)\to \pi_1(M')$ be an isomorphism. Suppose that $\phi$ satisfies the additional requirement that $\phi(\gamma)$ is an annular cusp loop in $M'$ if and only if $\gamma$ is an annular cusp loop in $M$. Then $\phi$ is induced by an isometry between $M$ and $M'$.
\end{st}

\noindent The additional requirement is necessary only in the case of 3--manifolds with non-compact geodesic boundary. In the compact case the set $\mathcal{A}_M$ is empty. 

In order to save the rigidity we need to include the annular cusp loops into the structure of the manifold itself. This will be done in the context of 3--manifolds with boundary pattern that were introduced in \cite{Johannson}.

\begin{df}
\label{df.ManifoldBoundaryPattern}
A 3--manifold with boundary pattern is a pair $(M,P)$ where $M$ is a 3--manifold with boundary and $P$ is a one dimensional polyhedron $P \subset \partial M$. A homeomorphism of manifolds with boundary patterns is required to restrict to a homeomorphism between the boundary patterns.
\end{df}

\noindent If $M$ is a hyperbolic 3--manifold with geodesic boundary then we would like to include the boundary circles of the annuli $\mathcal{A}_M$ in the natural compactification of $M$ as a boundary pattern but of course they are not part of $\partial M$. Since the annuli connect in $\bar{M}$ to $\partial{M}$ we can push them inside a little to become part of $\partial M$.

\begin{df}
\label{df.BoundaryPattern}
The boundary pattern corresponding to the hyperbolic structure with geodesic boundary on a $M$ is defined to be the set of boundary curves of the annuli in $\mathcal{A}_M$, pushed inside of $\partial M$.
\end{df}

\noindent A corollary of the above rigidity theorem is now the following:

\begin{st}
\label{st.Rigidity2}
Let $(M,P)$ and $(M',P')$ be two orientable finite volume hyperbolic 3--manifolds with geodesic boundary and let $P$ and $P'$ be their corresponding boundary patterns.
If $f:(M,P)\to (M',P')$ is a homeomorphism of 3--manifolds with boundary pattern then $f$ is induced by an isometry between $M$ and $M'$.
\end{st}

\noindent Thus the hyperbolic structure is still rigid if one takes into account the boundary patterns. Therefore we should only allow hyperbolic structures that agree with the given boundary pattern.

\begin{df}
\label{df.HypBoundary}
A 3--manifold with boundary pattern $(M,P)$ is said to allow a hyperbolic structure with geodesic boundary if it can be given a finite volume hyperbolic structure with geodesic boundary that turns the components of $P$ into annular cusp loops.
\end{df}

\noindent To define the volume for more general manifolds with boundary pattern we use the JSJ--decomposition and add up the volumes of the pieces allowing a hyperbolic structure with geodesic boundary. We state a version of the JSJ--decomposition for 3--manifolds with boundary pattern taken from \cite{Matveev}.

\begin{st}
\label{st.JSJ}
Let $(M,P)$ be an orientable, irreducible and boundary irreducible 3--manifold with boundary pattern. There exists a JSJ-system of annuli and tori that is unique up to admissible isotopy.
The system decomposes $(M,P)$ into three types of JSJ-chambers: simple 3--manifolds, Seifert manifolds and I--bundles.
\end{st}

\noindent Note that the JSJ--chambers are also 3--manifolds with boundary pattern. In addition to the original boundary pattern of $(M,P)$ they also inherit the adjacent boundary curves of the annuli in the JSJ--system \cite{Matveev}. 

\begin{df}
\label{df.HypVol}
Let $(M,P)$ be an orientable, irreducible and boundary irreducible 3--manifold with boundary pattern. We define the hyperbolic volume $ \mathrm{Vol}(M,P)$ of $(M,P)$ to be the sum of the hyperbolic volumes of the JSJ--chambers that allow a hyperbolic structure with geodesic boundary. 
\end{df}

\noindent The rigidity theorem (Theorem \ref{st.Rigidity2}) above and the uniqueness of the JSJ--decomposition show that the volume is a well defined invariant of orientable, irreducible and boundary irreducible 3--manifold with boundary pattern. The definition of volume can extended further by demanding it to be additive under connected sums. 

As a motivation for this definition of the hyperbolic volume of a 3--manifold with boundary pattern we note that it coincides with the simplicial volume in the case of an empty boundary \cite{Ratcliffe}. However for manifolds with boundary this notion seems to be more appropriate. Indeed the Gromov norm no longer agrees with the volume of a hyperbolic manifold as soon as the boundary is non-empty \cite{Jungreis}.

The most important example for our purposes is the so called outside of a graph. This is the version of the complement of a graph that is suitable for carrying a rigid hyperbolic structure.

\begin{df} 
\label{df.GraphComplement}
Let $\Gamma$ be an embedded graph in $\mathbb{S}^3$, where edges without vertices and multiple edges are allowed. We define the outside $O_\Gamma$ of $\Gamma$ to be the 3--manifold with boundary pattern constructed as follows.

Let $N(\Gamma)$ be the neighborhood of $\Gamma$ made up from small open balls around the vertices, closed solid tori around the edges of $\Gamma$ without vertices and small closed solid cylinders around the edges that intersect the closure of the balls around the adjacent vertices in disjoint disks. Define the outside $O_{\Gamma}$ be $\mathbb{S}^3-N(\Gamma)$. Also define the exterior $E_\Gamma$ to be the closure of $O_{\Gamma}$ as a subspace of $\mathbb{S}^3$. 

We will endow $O_\Gamma$ with the boundary pattern $P_\Gamma$ consisting of a circle around every hole on every holed sphere in its boundary. 
\end{df}

\noindent The outside of a graph may not be irreducible because the graph might be the distant union of a number of split components. If this is the case we cut the outside along spheres and cap the spheres off with balls. The resulting pieces are outsides of non-splittable graphs. For such graphs the outside is an orientable, irreducible and boundary irreducible 3--manifold whose boundary consists spheres from which closed disks have been removed. We have one sphere for every vertex the number of holes it has is equal to the valency of the vertex. 

The outside of a graph is not compact and neither is its boundary. The corresponding exterior is compact and will play the role of the natural compactification mentioned above. In the next section we will investigate the geometry and decomposition of 1--augmented KTGs in greater detail.

\subsection{The geometry of augmented KTGs}
\label{sub.GeometryAugmented}

In this subsection we prove part 2) of the main theorem. Let $\Gamma$ be an $n$--augmented KTG. The JSJ--system of the outside $O_\Gamma$ consists of tori only. One for every k--unzip move used to produce $\Gamma$, such that $k \geq 2$. The tori encircle the augmentation rings produced by the k--unzip move. Cutting along such a torus splits off a Seifert fibered JSJ--chamber of the form $(D_k \times \mathbb{S}^1, \emptyset)$, where $D_k$ is a $k$--times punctured disk and $k$ is the number of augmentation rings produced in the k--unzip move. After removing all such Seifert pieces we are left with a JSJ--chamber that is exactly the outside of the singly augmented KTG $\Gamma'$ corresponding to $\Gamma$. Note that by definition the hyperbolic volume of $\Gamma$ is equal to the volume of $\Gamma'$, since we neglect Seifert fibered chambers in the JSJ--composition. 

We aim to show that the outside of any singly augmented KTG $\Gamma'$ admits a hyperbolic structure with geodesic boundary by decomposing it into regular ideal octahedra. The method of decomposition is similar to the construction for links in \cite{FuterPurcell}.

\begin{figure}[here]
\begin{center}
\includegraphics[width = 12cm]{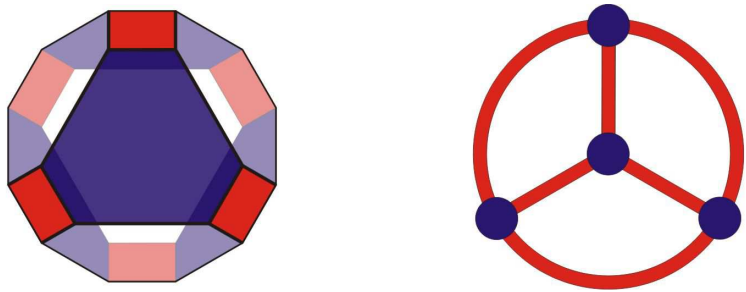}
\caption{A truncated octahedron with colored faces (left). The truncated octahedron as the half space beyond the paper plus infinity (right).}
\label{fig.Octahedron}
\end{center}
\end{figure}

\noindent The first step is to use truncated octahedra to create the exterior of the augmented KTG. The truncated octahedra we use are combinatorial closed polyhedra with eight hexagonal faces and six square truncation faces. Half of the hexagonal faces are colored blue, the other half white in an alternating fashion. The truncation faces are painted red, see figure \ref{fig.Octahedron} (left). 

\begin{lem}
\label{lem.Triangulation}
Every sequence of KTG moves $S$ has the following properties: 
\begin{enumerate}
\item{The exterior $E_S$ of the singly augmented KTG $\Gamma_S'$ is homeomorphic to the space obtained by glueing together $2t+2$ truncated octahedra, where $t$ is the number of triangle moves in $S$.} 
\item{If $\beta$ is a sufficiently small ball around an interior point of an edge in $E_S$, the intersection of $\beta$ and the union of the interiors of the octahedra making up $E_S$ has either two or four components.}
\item{For each vertex of $\Gamma_S'$ there is a pair of blue faces that is sent by the homeomorphism from part 1) onto the three holed sphere in the boundary of the exterior of $\Gamma_S'$ corresponding to that vertex. The boundary circles of every hole are glued together from pairs of red edges of the two blue faces.}
\end{enumerate}
\end{lem}

\begin{proof}
The proof proceeds by induction on the number of KTG moves in the sequence $S$. 

\textbf{Induction basis.} Let us suppose first that $S$ is empty so that $\Gamma_S'$ is the standard tetrahedron. Now take two truncated octahedra and glue their white faces together in pairs via the identity. To see how this produces the exterior of the tetrahedron graph let us first look at a single truncated combinatorial octahedron, see figure 11 (right). 

By a homeomorphism we can present the truncated octahedron as the upper half space (thought of as lying behind the paper) plus infinity with the colored faces on its boundary. The blue faces are now small blue disks in the plane, while one white face is stretched out so as to contain infinity.  Now we bend the blue and red faces up as in figure \ref{fig.Base}. In this figure the interior of the octahedron is located directly above the blue dome-like faces in the upper half space. The horizontal plane on which the blue domes rest contains the white faces. 

We can place the second octahedron in the lower half space with the blue faces pushed downwards so that it looks like the reflection of the upper octahedron in the horizontal plane. Glueing the octahedra together along the white faces thus produces a 3--manifold homeomorphic to the exterior of the tetrahedral graph. Since we used exactly two octahedra part 1) is proven.

For part 2) note that all edges of the exterior are alike so that we can concentrate on one of them. A small ball around an interior point of such an edge intersected with the interiors of the octahedra has two components: one in the upper half space and one in the lower half space.

The third part is also clear since the exterior can be arranged in such a way that the horizontal plane cuts it into mirror symmetrically arranged truncated octahedra.

\begin{figure}[h]
\begin{center}
\includegraphics[width = 12cm]{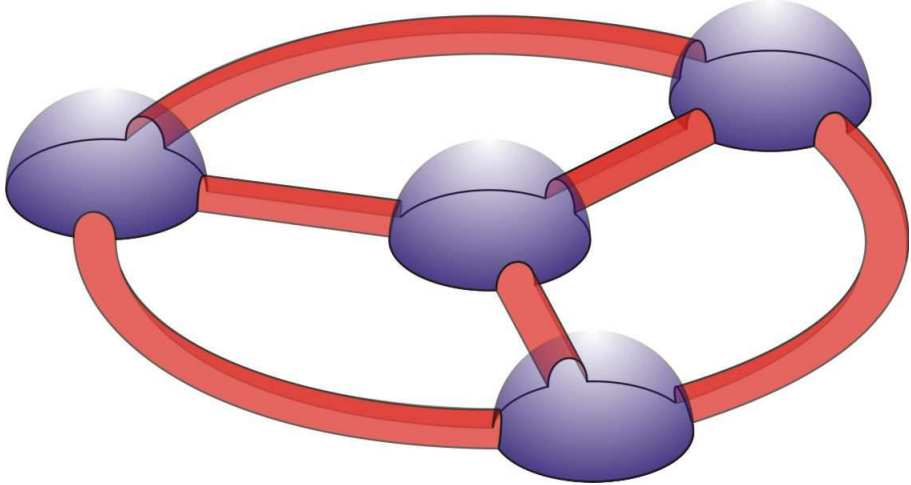}
\caption{The exterior of the standard tetrahedron can be obtained by glueing two truncated octahedra. Only the upper one is shown.}
\label{fig.Base}
\end{center}
\end{figure}

\textbf{Induction step.} Suppose $S$ is a sequence of KTG moves that has the properties 1,2,3 in the lemma. Let $T$ be a sequence of KTG moves obtained by performing one of the four KTG moves directly after $S$. In order to show that $T$ also has properties 1,2,3 we need to consider four cases depending on which KTG move was made: negative or positive half twist, triangle or unzip.

Half twist. If the last move was a half twist then the exteriors of $\Gamma_S$ and $\Gamma_T$ are homeomorphic and the number of triangle moves in producing them is equal. We can therefore use the gluing of truncated octahedra that worked for $S$.

Triangle move. Since $T$ contains one more triangle move than $S$ we need two more truncated combinatorial octahedra to glue the exterior $E_T$ than we needed to glue $E_S$. We will call the two new truncated octahedra $O_1$ and $O_2$. Let $v$ be the vertex of $\Gamma_S$ where the triangle move was performed. By the induction hypothesis 3) we know there are two blue faces $B_1$ and $B_2$ in $E_S$ that make up the three-holed sphere corresponding to $v$. The new glueing is produced from the old by decreeing that one blue face of $O_i$ is to be identified with the face $B_i$. The corresponding pairs of white faces of $O_1$ and $O_2$ should be identified also. 

To see that that the exterior of $\Gamma_T$ is homeomorphic to the above gluing we start by bringing the truncated octahedron $O_1$ into the dome-like form seen in figure \ref{fig.Triangle} (right). The chosen blue face (drawn slightly transparent) is a hemisphere and the rest of $O_1$ is below it. The other blue faces are small domes and the white faces are horizontal. The red faces are half tubes. Now bring $O_2$ into mirror symmetric position below the horizontal plane and glue them along the white faces. The result is a closed ball with three tubular entrances connecting to a triangular tunnel in the middle. It is now clear that once we glue this ball inside the three-holed sphere corresponding to the vertex $v$ we get the exterior of $\Gamma_T$.

To check property 2) we only need to check the edges of $O_1$ and $O_2$. For them it is clear from the mirror symmetric arrangement of $O_1$ and $O_2$. This also proves property 3).

\begin{figure}[h]
\begin{center}
\includegraphics[width = 12cm]{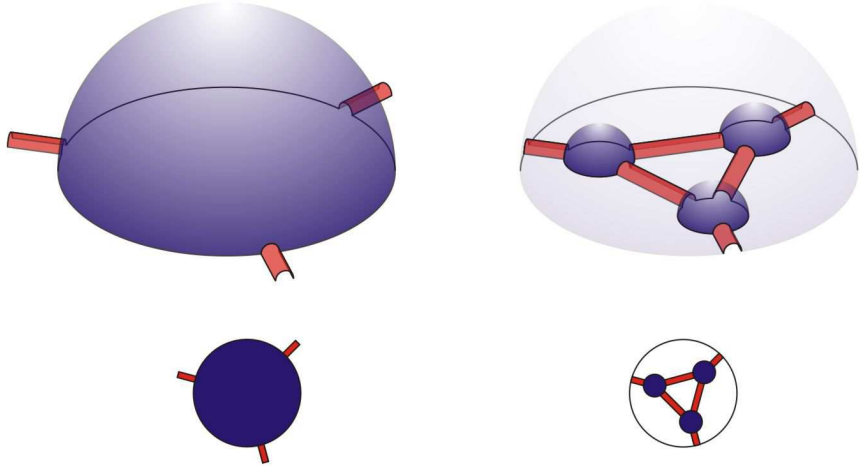}
\caption{The operation corresponding to the triangle move. Only the boundary of the upper half has been depicted. The smaller pictures are a top view and should remind one of the KTG move.}
\label{fig.Triangle}
\end{center}
\end{figure}

Unzip. We will show that the gluing of octahedra that produces the exterior $E_S$ also produces $E_T$ after adding an extra identification of faces. Suppose that $\Gamma_T$ is obtained from $\Gamma_S$ by performing a 1--unzip move on the edge $e$. Take a small open ball neighborhood $B$ in $\mathbb{S}^3$ of the tube around $e$ that contains the two three-holed spheres around to the endpoints of $e$ but does not meet any
other parts of the boundary of $E_S$. This ball is depicted as a cylinder in figure \ref{fig.Unzip} (upper left). By the induction hypothesis we know that under the homeomorphism from part 1) the three-holed spheres both split up into two blue faces each in such a way that the boundary circles are glued from pairs of edges. One can thus arrange the ball $B$ in $\mathbb{R}^3$ such it is mirror symmetric with respect to the horizontal plane. Cutting along the horizontal plane produces two balls $B_1$ and $B_2$. The boundary of one of the balls can then be flattened to look like the second picture of figure 14 (upper right).

Now let us glue together the two blue faces in $B_1$. This produces the next picture in figure \ref{fig.Unzip} (lower right). The red face in the middle of the second picture becomes a tube and the opposite red faces are joined. Now glue the blue faces of $B_2$ in the same way and then glue $B_1$ and $B_2$ back together. We get the ball $B'$ seen in the last picture of figure \ref{fig.Unzip} (lower left). 

Note that when there are $h$ half twists present on the edge $e$ then performing an unzip produces $h$ half twists between the resulting strands. To accommodate this feature in our glueing we cut $B'$ open again along the two pairs of blue faces. They form a punctured disk whose boundary circle is a longitude of the newly produced ring. The disk is pierced twice by the two horizontal components of the graph that go through the ring. A half twist in these two components is produced by reglueing the disks with a half twist. This last correction gives the homeomorphism between the exterior $E_T$ and the glueing of octahedra. When $h$ is even then the correction has not changed the glueing, but if $h$ is odd then we have identified the blue faces of $B_1$ to the diagonally opposite ones of $B_2$.

\begin{figure}[h]
\begin{center}
\includegraphics[width = 12cm]{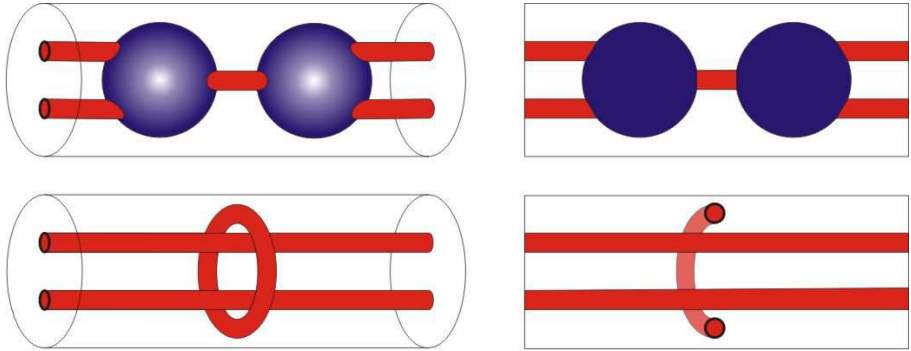}
\caption{The homeomorphism corresponding to the 1--unzip.}
\label{fig.Unzip}
\end{center}
\end{figure}

The extra identification of faces has doubled the number of parts of octahedra coming together at some of the edges of the blue faces involved, but these were previously unglued so this settles part 2) of the lemma. Part 3) is still true because we simply deleted two vertices and left the exterior unchanged around the other ones.
\end{proof}

\noindent Now that we have constructed the exterior of the singly augmented KTG $\Gamma'$ the next step is to go back to its outside. 

\begin{lem}
\label{lem.Glueing}
The outside $O_{\Gamma'}$ is homeomorphic as a 3--manifold with boundary pattern to the glueing of truncated octahedra that we constructed for the exterior in lemma \ref{lem.Triangulation}, except that we remove all closed red square faces and endow it with the boundary pattern formed by lines on the unglued blue faces that are parallel to the removed red edges. 
\end{lem}
\begin{proof}
The proof of the previous lemma goes through step by step if we replace the exterior by the outside and remove the red truncation faces. It is easy to see that the homeomorphism can be made to identify the boundary patterns.
\end{proof}

\noindent Finally we turn to hyperbolic geometry. The above gluing of truncated octahedra has
the property that at each edge either two or four solid angles meet. This means that if we declare
all the truncated octahedra to be regular ideal hyperbolic
octahedra then we obtain a hyperbolic manifold with geodesic boundary with cusps based on the tori and annuli that used to be truncation faces \cite{Ratcliffe}.

The circles in the boundary pattern of the outside $O_{\Gamma'}$ now become annular cusp loops because in the exterior they are freely homotopic to the boundary circles of the annuli in the closure of $O_{\Gamma'}$. The exterior is exactly the natural compactification of $O_{\Gamma'}$. This finishes the construction of the hyperbolic structure on the outside of a singly augmented KTG and also the proof part 2) of the main theorem.

\section{Conclusion}
\label{sec.Conclusion}

The purpose of this paper was to generalize the volume conjecture to KTGs and to prove it for augmented KTGs. In order to generalize to links and KTGs it was necessary to restrict to odd colors. For knots this seems unnecessary and in general one may ask which KTGs will satisfy the original volume conjecture. 

The generalization of the volume of the complement to KTGs involved considering a specific 3--manifold with boundary pattern called the outside of a graph. This notion also makes sense for arbitrary graphs so that one may try to apply geometric techniques to questions in graph embedding. One may also hope to generalize the volume conjecture to arbitrary graphs, but then one must first be able to define the colored Jones invariant of any vertex. For trivalent vertices the colored Jones invariant has a natural meaning as a Clebsch--Gordan projector but for arbitrary vertices there is more choice. 

In this paper we have proven the volume conjecture for augmented KTGs provided they had sufficiently many augmentation rings. It would be very natural to try to remove this restriction on the number of rings but this will require a more detailed analysis of the colored Jones invariant of such KTGs.

Looking back we can summarize our proof as follows. We have seen three different meanings of the KTG moves. Firstly they can be used to generate all KTGs from the tetrahedron. Secondly, reading them backwards yields an expression for the colored Jones invariant in terms of six-j symbols. Thirdly the augmented moves encode combinatorics of the triangulation by octahedra of the corresponding singly augmented KTG. The second and the third viewpoint come together once one notices that augmenting kills the summations in the expression for the Jones invariant (at least at the root of unity). Using the known asymptotics of the regular six-j symbol that remains this gives a natural and proof for the volume conjecture for augmented KTGs.
 
It seems that the augmented KTGs form a tractable class of KTGs that makes a good testing ground for further extensions of the volume conjecture. For example the complexified volume conjecture \cite{MMOTY}. It is to be hoped that with the right definition of the Chern--Simons invariant for manifolds with boundary pattern this conjecture also holds for KTGs. 

A reason for the tractability of the augmented KTGs might be that they are of arithmetic type, see corollary \ref{cor.Arithmetic}. So far all knots links and KTGs for which the volume conjecture was proven, were of arithmetic type or not hyperbolic at all (or a combination of the two).

\end{document}